\documentclass[a4paper,12pt]{article}

\usepackage[textwidth=125mm, textheight=195mm]{geometry}
\usepackage[english]{babel}
\usepackage{graphicx}
\usepackage{amsmath}
\usepackage{amsfonts}
\usepackage{amsthm}
\usepackage{epstopdf}

\begin{document}

\newcommand{\wk}{\mbox{$\,<$\hspace{-5pt}\footnotesize )$\,$}}

\numberwithin{equation}{section}
\newtheorem{teo}{Theorem}
\newtheorem{lemma}{Lemma}

\newtheorem{coro}{Corollary}
\newtheorem{prop}{Proposition}
\theoremstyle{definition}
\newtheorem{definition}{Definition}
\theoremstyle{remark}
\newtheorem{remark}{Remark}

\newtheorem{scho}{Scholium}
\newtheorem{open}{Question}
\newtheorem{example}{Example}
\numberwithin{example}{section}
\numberwithin{lemma}{section}
\numberwithin{prop}{section}
\numberwithin{teo}{section}
\numberwithin{definition}{section}
\numberwithin{coro}{section}
\numberwithin{figure}{section}
\numberwithin{remark}{section}
\numberwithin{scho}{section}

\bibliographystyle{abbrv}

\title{The Rosenthal-Szasz inequality for normed planes}
\date{}

\author{Vitor Balestro  \\ CEFET/RJ Campus Nova Friburgo \\ 28635000 Nova Friburgo \\ Brazil \\ vitorbalestro@gmail.com \and Horst Martini \\ Fakult\"{a}t f\"{u}r Mathematik \\ Technische Universit\"{a}t Chemnitz \\ 09107 Chemnitz\\ Germany \\ martini@mathematik.tu-chemnitz.de}

\maketitle

\begin{abstract} We aim to study the classical Rosenthal-Szasz inequality for a plane whose geometry is given by a norm. This inequality states that the bodies of constant width have the largest perimeter among all planar convex bodies of given diameter. In the case where the unit circle of the norm is given by a Radon curve, we obtain an inequality which is completely analogous to the Euclidean case. For arbitrary norms we obtain an upper bound for the perimeter calculated in the anti-norm, yielding an analogous characterization of all curves of constant width. To derive these results, we use methods from the differential geometry of curves in normed planes.

\end{abstract}


\bigskip


\section{Introduction}
The classical \emph{Rosenthal-Szasz theorem} (see \cite{Ro-Sz}, \cite[Section 44]{BF}, \cite[p. 143]{Buch}, and \cite[p. 386]{Mos}) says that for a compact, convex figure $K$ in the Euclidean plane with perimeter $p(K)$ and diameter $D(K)$ the inequality $p(K) \leq \pi D(K)$ holds, with equality if and only if $K$ is a planar set of constant width $D(K)$. That \emph{any} figure of the same constant width satisfies the equality case is clear by Barbier's theorem (cf. \cite{Barb} and \cite[Section 44]{BF}), saying that all convex figures of fixed constant width have the same perimeter. We will extend the Rosenthal-Szasz theorem to all normed planes whose unit circle is a Radon curve. A basic reference regarding the geometry of normed planes and spaces is the monograph \cite{Tho}, whereas the important subcase of Radon planes is comprehensively discussed in \cite{Ma-Swa}. For switching to normed planes, we modify the notation above slightly.

Our approach to Radon planes follows \cite{Ma-Swa}, and we start by introducing an orthogonality concept. We say that a vector $v \in X$ is (\emph{left}) \emph{Birkhoff orthogonal} to $w$ (denoted by $v \dashv_B w$) if $||v|| \leq ||v+\lambda w||$ for any $\lambda \in \mathbb{R}$. If $v$ and $w$ are non-zero vectors, this is equivalent to stating that the unit ball is supported at $v/||v||$ by a line whose direction is $w$. Consequently, if $v \dashv_Bw$, then the distance from any fixed point of the plane to a line in the direction $w$ is attained by a segment whose direction is $v$. This property of Birkhoff orthogonality will be important later. For the notion of Birkhoff orthogonality, we refer the reader to \cite{alonso} and \cite[$\S$ 3.2]{Tho}. 
Fixing a symplectic bilinear form $\omega$ on $X$ yields an identification between $X$ (which is unique up to constant multiplication) and its dual $X^*$:
\begin{align*} X \ni x \mapsto \iota_x\omega = \omega(x,\cdot) \in X^*,
\end{align*}
and this allows us to identify the usual dual norm in $X^*$ with a norm in $X$ given as
\begin{align*} ||y||_a := \sup\{\omega(x,y): x \in B\}, \ \ y \in X,
\end{align*}
where $B = \{x \in X: ||x|| \leq 1\}$ is the \emph{unit ball} of $(X,||\cdot||)$. This norm is called the \emph{anti-norm}. A \emph{Radon plane} is a normed plane in which Birkhoff orthogonality is a symmetric relation. Since the anti-norm reverses Birkhoff orthogonality, this is equivalent to the statement that the anti-norm is a multiple of the norm. Consequently, in a Radon plane one may assume, up to rescaling the symplectic form $\omega$, that $||\cdot||_a = ||\cdot||$, and \emph{we will always assume that this is assured}. 

Let $(X,||\cdot||)$ be a normed plane where (at least at the beginning of our argumentation) the unit circle $S:=\{x \in X: ||x|| = 1\}$ is a closed, simple and convex curve of class $C^2$. With a fixed norm, one can define the length of a curve $\gamma:[a,b]\rightarrow X$ in the usual way by
\begin{align*}  l(\gamma) := \sup\left\{\sum_{j=1}^n||\gamma(t_{j})-\gamma(t_{j-1})||: a = t_0 < t_1 <\ldots < t_n = b \ \mathrm{is \ a \ partition \ of} \ J\right\},
\end{align*}
and whenever $\gamma$ is smooth, one clearly has
\begin{align*} l(\gamma) = \int_{a}^b||\gamma'(t)||\ dt.
\end{align*}

Also, the choice of a symplectic bilinear form yields an area element and an orientation. An important feature of Radon planes is that the \emph{Kepler law} holds for them: the arc-length of the unit circle is proportional to the area of the corresponding sector of the unit ball. In the normalization that we are adopting, the arc-length is actually twice the area of the sector. Equivalently, if $\varphi:\mathbb{R} \ \mathrm{mod} \ l(S) \rightarrow X$ is a positively oriented parametrization of the unit circle by arc-length, then $\omega(\varphi,\varphi') = ||\varphi'|| = 1$. 

If $\gamma$ is a closed, simple and convex curve, then its range $\{\gamma\}$ bounds a planar convex body $K_{\gamma}$. It is well known that for each direction $v$ of $X$ this convex body is supported by two parallel lines orthogonal (in the Euclidean sense) to $v$. The distance (in the norm) between these two supporting lines is the (\emph{Minkowski}) \emph{width} of $K_{\gamma}$ in the direction of $v$. If this number is independent of the direction $v$, then we say that the convex body $K_{\gamma}$ (or the curve $\gamma$) has \emph{constant width} (see \cite[$\S$ 4.2]{Tho}).  Also, we define the \emph{diameter} of a given subset $A \subseteq X$ to be
\begin{align*} \mathrm{diam}(A) = \sup\{||x-y||:x,y \in A\}.
\end{align*}
Our objective in this note is to prove the following theorem.

\begin{teo}\label{main} Let $\gamma:S^1\rightarrow X$ be a closed, simple, convex curve in a Radon plane $(X,||\cdot||)$. Then
\begin{align*} l(\gamma) \leq \mathrm{diam}\{\gamma\}\cdot\frac{l(S)}{2},
\end{align*}
and equality holds if and only if $\gamma$ is a curve of constant width $\mathrm{diam}\{\gamma\}$. 
\end{teo}

This is clearly an extension of the classical Rosenthal-Szasz inequality holding for the Euclidean plane (see \cite{Ro-Sz}). It is worth mentioning that this inequality was already studied to $2$-dimensional spaces of constant curvature in \cite{cifre}.

The main ``tool" that we use here is the differential geometry of smooth curves in normed planes, for which our main reference is \cite{Ba-Ma-Sho} (see also \cite{Ma-Wu}). Based on this, we will prove the result for smooth curves in smooth normed planes and, after that, we can extend the result to the non-smooth case by routine approximation arguments. For a given $C^2$ curve $\gamma(s):[0,l(\gamma)]\rightarrow X$ parametrized by arc-length $s$, choose a smooth function $t:[0,l(\gamma)]\rightarrow \mathbb{R}$ such that
\begin{align*} \gamma'(s) = \frac{d\varphi}{dt}(t(s)),
\end{align*} 
where we recall that $\varphi(t)$ is a positively oriented parametrization of the unit circle by arc-length $t$. Geo\-me\-tri\-cally, we are identifying where the (oriented) line in the direction of $\gamma'(s)$ supports the unit ball $B := \{x \in X:||x|| \leq 1\}$. The \emph{circular curvature} of $\gamma$ at $\gamma(s)$ is the number
\begin{align*} k_{\gamma}(s) := t'(s).
\end{align*}
In any point of $\gamma$, where $k_{\gamma}(s) \neq 0$, the number $\rho(s):=k_{\gamma}(s)^{-1}$ is the \emph{radius of curvature} of $\gamma$ at $\gamma(s)$. This is the radius of an osculating circle of $\gamma$ at $\gamma(s)$. For curves of constant width the following was first settled for \emph{any} smooth Minkowski plane (not necessarily Radon) in \cite{Pet} (see also \cite{Ba-Ma-Sho} for an elegant proof).

\begin{prop} Let $\gamma:S^1\rightarrow X$ be a simple, closed, strictly convex curve of class $C^2$ having constant width $d$. Then we have: \\

\noindent\textbf{\emph{(a)}} The sum of the curvature radii at any pair of points of $\gamma$ belonging to two parallel supporting lines of $K_\gamma$ equals $d$.\\

\noindent\textbf{\emph{(b)}} The length of $\gamma$ satisfies the equality $l(\gamma) = d\cdot\frac{l(S)}{2}$. 
\end{prop}

Part \textbf{(b)} is the extension of Barbier's theorem to Minkowski planes, see (in addition to \cite{Pet} and \cite{Ba-Ma-Sho}) also \cite{chai}, \cite{Mi-Mo}, and \cite{Ma-Mu}.

\section{Proof of Theorem \ref{main}}

Let $\gamma:S^1\rightarrow X$ be a simple, closed, strictly convex curve of class $C^2$. For the sake of simplicity, we assume also that the region bounded by $\{\gamma\}$ contains the origin. We define the \emph{Minkowski support function} of $\gamma$ to be the function which associates each point $\gamma(s)$ to the distance $h_{\gamma}(s)$ from the support line of $K_{\gamma}$ at $\gamma(s)$ to the origin. With that definition, for any $s_0,s_1 \in S^1$ such that $K_{\gamma}$ is supported at $\gamma(s_0)$ and $\gamma(s_1)$ by parallel lines, we get the inequality
\begin{align}\label{suppdiam} h_{\gamma}(s_0) + h_{\gamma}(s_1) \leq ||\gamma(s_0)-\gamma(s_1)|| \leq \mathrm{diam}\{\gamma\}.
\end{align}
This comes from the fact (described above) that the distance between the support lines must be attained by a segment which is in the left Birkhoff orthogonal direction to them (see Figure \ref{minksupp}). 

\begin{figure}[h]
\centering
\includegraphics{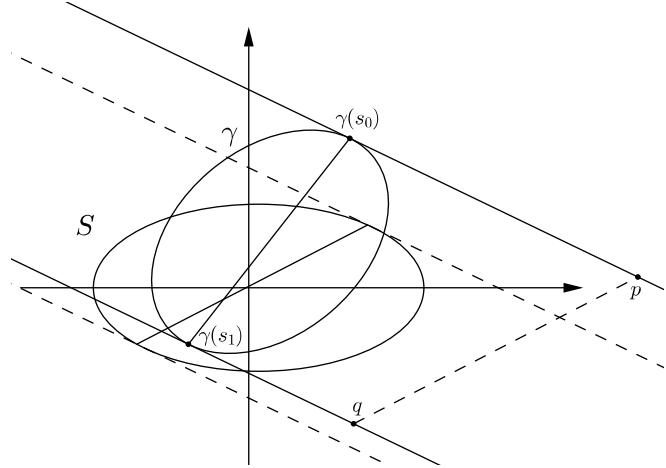}
\caption{$||p-q|| = h_{\gamma}(s_0)+h_{\gamma}(s_1) \leq ||\gamma(s_0) - \gamma(s_1)|| \leq \mathrm{diam}\{\gamma\}$.}
\label{minksupp}
\end{figure}

Recall that $\varphi:\mathbb{R} \ \mathrm{mod} \ l(S) \rightarrow \mathbb{R}$ is a positively oriented parametrization of the unit circle by arc-length, and remember also that we have $\omega(\varphi,\varphi') = 1$. Assume now that $\gamma$ is endowed with a parameter $u$ for which
\begin{align*} \gamma'(u) = f(u)\cdot\varphi'(u), \ \ u \in [0,l(S)],
\end{align*}
with $f > 0$ (i.e., the parametrization is positively oriented). We get immediately that $f = \omega(\varphi,\gamma')$. For each $u$, we decompose $\gamma$ in the basis $\{\varphi(u),\varphi'(u)\}$ to obtain
\begin{align*} \gamma = \omega(\gamma,\varphi')\varphi - \omega(\gamma,\varphi)\varphi'.
\end{align*}
Since the support line to $K_{\gamma}$ at $\gamma(u)$ has the direction $\varphi'(u)$, we have that the distance of this line to the origin is simply the projection of $\gamma(u)$ in the direction $\varphi(u)$. It follows from the equality above that the support function of $\gamma$ is given by
\begin{align*} h_{\gamma}(u) = \omega(\gamma(u),\varphi'(u)).  
\end{align*}
Now we calculate the length of $\gamma$:
\begin{align*} l(\gamma) = \int_0^{l(S)}||\gamma'(u)|| \ du = \int_{0}^{l(S)}\omega(\varphi(u),\gamma'(u))\cdot||\varphi'(u)|| \ du = \\ = \int_0^{l(S)}\omega(\varphi(u),\gamma'(u)) \ du = \int_0^{l(S)}\omega(\varphi(u),\gamma(u))' - \omega(\varphi'(u),\gamma(u)) \ du = \\ = \int_0^{l(S)}\omega(\gamma(u),\varphi'(u)) \ du = \int_0^{l(S)}h_{\gamma}(u) \ du.
\end{align*}
Now notice that the support lines to $K_{\gamma}$ at $\gamma(u)$ and $\gamma\left(u+l(S)/2\right)$ are parallel. Hence, from (\ref{suppdiam}) we get
\begin{align}\label{supp} h_{\gamma}(u) + h_{\gamma}\left(u+l(S)/2\right) \leq \mathrm{diam}\{\gamma\}.
\end{align}
Finally,
\begin{align*} l(\gamma) = \int_0^{l(S)}h_{\gamma}(u) \ du = \int_0^{l(S)/2}h_{\gamma}(u) + h_{\gamma}\left(u+l(S)/2\right) \ du \leq \mathrm{diam}\{\gamma\}\cdot\frac{l(S)}{2},
\end{align*}
and equality holds if and only if equality holds in (\ref{supp}) for each $u$. This clearly characterizes bodies of constant width.

This proves the theorem for the case that both $S$ and $\gamma$ are of class $C^2$. The general case follows from standard approximation of bodies which are not necessarily smooth or strictly convex by bodies whose boundaries are of class $C^2$ in the Hausdorff metric (it is not hard to check that a non-smooth Radon curve can be approximated by smooth Radon curves). Since all quantities involved in the inequality are clearly continuous in that metric, we have the result. 

\section{The general case}

For the sake of completeness, in this section we briefly explain why our approach does not work for general normed planes. We also provide a weaker bound, which is enough to prove Barbier's theorem. This discussion yields a nice characterization of constant width curves in arbitrary normed planes, where the anti-norm is involved. With the same notation as in the previous section we have that, if the plane is not Radon, $\omega(\varphi,\varphi')$ is not a constant function.  Hence, in the parametrization $\gamma(u)$ such that $\gamma'(u) = f(u)\cdot\varphi'(u)$, we get 
\begin{align*} f = \frac{\omega(\varphi,\gamma')}{\omega(\varphi,\varphi')},
\end{align*} 
where the parameter was omitted to simplify the notation. Also, the decomposition of $\gamma(u)$ in the basis $\{\varphi(u),\varphi'(u)\}$ now reads
\begin{align*} \gamma = \frac{\omega(\gamma,\varphi')}{\omega(\varphi,\varphi')}\varphi - \frac{\omega(\gamma,\varphi)}{\omega(\varphi,\varphi')}\varphi',
\end{align*}
from where the support function of $\gamma$ is given by
\begin{align*} h_{\gamma} = \frac{\omega(\gamma,\varphi')}{\omega(\varphi,\varphi')}.
\end{align*}
Therefore, calculating the length of $\gamma$ we obtain the following bound:
\begin{align*} l(\gamma) = \int_{0}^{l(S)}\frac{\omega(\varphi,\gamma')}{\omega(\varphi,\varphi')} \ du = \int_0^{l(S)} \frac{\omega(\varphi,\gamma)'}{\omega(\varphi,\varphi')} \ du + \int_0^{l(S)} \frac{\omega(\gamma,\varphi')}{\omega(\varphi,\varphi')} \ du = \\ = \int_0^{l(S)} \frac{\omega(\varphi,\gamma)'}{\omega(\varphi,\varphi')} \ du + \int_0^{l(S)}h_{\gamma} \ du \leq \mathrm{diam}\{\gamma\}\cdot\frac{l(S)}{2} + \int_0^{l(S)} \frac{\omega(\varphi,\gamma)'}{\omega(\varphi,\varphi')} \ du,
\end{align*}
and the last integral does not necessarily vanish. As mentioned before, this weaker inequality can be used to prove Barbier's theorem. Indeed, denoting for simplicity $\pi := l(S)/2$, if $\gamma$ is a curve of constant width, then $\gamma(u) - \gamma(u+\pi)$ points in the direction $\varphi(u)$, and we get
\begin{align*} \int_0^{l(S)}\frac{\omega(\varphi,\gamma)'}{\omega(\varphi,\varphi')} \ du = \int_0^{\pi}\frac{\omega(\varphi(u),\gamma(u)-\gamma(u+\pi))'}{\omega(\varphi(u),\varphi'(u))} \ du = 0,
\end{align*}
since $\omega(\varphi(u),\gamma(u)-\gamma(u+\pi))$ vanishes for every $u$. It remains an open problem whether the Rosenthal-Szasz inequality, as stated in Theorem \ref{main}, holds for all normed planes. However, for arbitrary normed planes we can derive the following bound for the perimeter in the anti-norm.

\begin{teo} Let $\gamma:S^1\rightarrow (X,||\cdot||)$ be a simple, closed and convex curve in an arbitrary normed plane. Then, denoting the perimeter in the anti-norm by $l_a$, we have
\begin{align*} l_a(\gamma) \leq \mathrm{diam}\{\gamma\}\cdot\frac{l_a(S)}{2},
\end{align*}
where we recall that $\mathrm{diam}\{\gamma\}$ is the diameter of $\gamma$ in the norm. Equality holds if and only if $\gamma$ is a curve of constant width (in the norm). 
\end{teo}
\begin{proof} We have the equality $\omega(\varphi,\varphi') = ||\varphi||\cdot||\varphi'||_a = ||\varphi'||_a$ (cf. \cite{Ma-Swa}). In the same notation as in the proof of Theorem \ref{main}, we calculate
\begin{align*} l_a(\gamma) = \int_0^{l(S)}||\gamma'(u)||_a \ du = \int_0^{l(S)}f(u)\cdot||\varphi'||_a \ du = \int_0^{l(S)}\omega(\varphi,\gamma') \ du = \\ = \int_0^{l(S)} \omega(\gamma,\varphi') \ du = \int_0^{l(S)}||\varphi'||_a\cdot\frac{\omega(\gamma,\varphi')}{\omega(\varphi,\varphi')} \ du = \int_0^{l(S)} ||\varphi'||_a\cdot h_{\gamma}(u) \ du = \\ = \int_0^{l(S)/2}||\varphi'(u)||_a\left(h_{\gamma}(u)+h_{\gamma}(u+l(S)/2)\right) \ du \leq \mathrm{diam}\{\gamma\}\int_0^{l(S)/2}||\varphi'(u)||_a \ du = \\ = \mathrm{diam}\{\gamma\}\cdot\frac{l_a(S)}{2}.
\end{align*}
Of course, equality holds if and only if $h_{\gamma}(u) + h_{\gamma}(u+l(S)/2) = \mathrm{diam}\{\gamma\}$ for any $u$. As we already mentioned, this characterizes curves of constant width. 

\end{proof}

It is clear that Theorem \ref{main} is a consequence of the previous result, since in a Radon plane (with our normalization) the anti-norm equals the norm. Also, since the anti-norm of the anti-norm is the original norm (see \cite{Ma-Swa}), we can bound the length of a convex curve in an arbitrary norm in terms of the respective anti-norm as follows.

\begin{coro} If $\mathrm{diam}_a$ denotes the diameter calculated in the anti-norm, then we have the inequality
\begin{align*} l(\gamma) \leq \mathrm{diam}_a\{\gamma\}\cdot\frac{l(S)}{2},
\end{align*}
and equality holds if and only if the curve $\gamma$ has constant width in the anti-norm. 
\end{coro}

\end{document}